\documentclass[12pt,oneside,reqno]{amsart}
\usepackage[all]{xy}
\usepackage{amsfonts,amsmath,oldgerm,amssymb,amscd}
\UseComputerModernTips
\numberwithin{equation}{section}
\usepackage[breaklinks]{hyperref}
\allowdisplaybreaks

\usepackage{enumerate}
\usepackage{caption} 
\newtheorem{theorem}{Theorem}[section]

\newtheorem{lemma}[subsection]{{\bf Lemma}}

\begin{document}

\title[Repdigits as difference of two balancing or Lucas-balancing numbers]{Repdigits as difference of two balancing or Lucas-balancing numbers} 

\author[M. Mohapatra]{M. Mohapatra}
\address{Monalisa Mohapatra, Department of Mathematics, National Institute of Technology Rourkela, Odisha-769 008, India}
\email{mmahapatra0212@gmail.com}

\author[P. K. Bhoi]{P. K.  Bhoi}
\address{Pritam Kumar Bhoi, Department of Mathematics, National Institute of Technology Rourkela, Odisha-769 008, India}
\email{pritam.bhoi@gmail.com}

\author[G. K. Panda]{G. K. Panda}
\address{Gopal Krishna Panda, Department of Mathematics, National Institute of Technology Rourkela, Odisha-769 008, India}
\email{gkpanda\_nit@rediffmail.com}


\begin{abstract}
Repdigits are natural numbers formed by the repetition of a single digit. In this paper, we study the problem of writing repdigits as the difference of two balancing or Lucas-balancing numbers. The method of proof involves the application of Baker's theory for linear forms in logarithms of algebraic numbers and the Baker-Davenport reduction procedure. Computations are done with the help of a simple computer program in {\it Mathematica}.
\end{abstract}

\maketitle
\pagenumbering{arabic}
\pagestyle{headings}

\section{Introduction}

The balancing number sequence $(B_n)_{n\geq0}$ and the Lucas-balancing sequence $(C_n)_{n\geq0}$ are defined by the binary recurrences
	\begin{equation}\label{bn}
		B_{n+1}=6B_n-B_{n-1},\quad B_0 = 0,\quad B_1 = 1
	\end{equation}
	and
	\begin{equation}\label{cn}
		C_{n+1} = 6C_{n}-C_{n-1},\quad C_0=1,\quad C_1=3.
	\end{equation}
	The Binet’s formula for the sequences are given by
	\begin{center}
		$B_n=\frac{\alpha^{n}-\beta^{n}}{4\sqrt{2}}$
		and 
		$C_n=\frac{\alpha^{n}-\beta^{n}}{2};$ for $n\geq1$,
	\end{center}
	where $(\alpha,\beta)=(3+2\sqrt{2},3-2\sqrt{2})$ is the pair of roots of the characteristic polynomial $x^2-6x-1$. Furthermore, it can be noted that $5 < \alpha < 6, \quad 0 < \beta < 1$ and we can prove by induction that 
	\begin{center}
		$\alpha^{n-1} \leq B_n <\alpha^n$ and $\alpha^{n} < 2C_n <\alpha^{n+1}$,
	\end{center} 
	holds for all $n\geq1$.
	\newline \newline  A repdigit is a positive integer $N$ that has only one distinct digit when written in base 10. Mathematically, it is in the form $d\biggr(\dfrac{10^{k}-1 }{9}\biggr)$, where $d\in\{1,2,\ldots,9\}$. For $k=1$, it is a trivial repdigit.
	\newline\newline Recently, searching of repdigits in linear recurrence sequences have been seen in several papers. Adegbindin et. al \cite{alt2019} determined all Lucas numbers that are sums of two repdigits. MG Duman \cite{d2023} has found all Padovan numbers which can be expressed as difference of two repdigits. Erduvan et. al \cite{ekl2021} found all Fibonacci and Lucas numbers which can be written as a difference of two repdigits.  F. Luca \cite{l2000} has found all Fibonacci and Lucas numbers which are repdigits.   S. G. Rayaguru and G. K. Panda \cite{rp2021} investigated the existence of all balancing and Lucas-balancing numbers that are expressible as the sums of two repdigits. The authors showed that 35 is the only balancing number which is concatenation of two repdigits in \cite{rp2020}. In \cite{rp2018}, the authors solved  the problem of finding the repdigits as product of consecutive balancing and Lucas-balancing numbers. K. Bhoi and P.K. Ray \cite{br2021} proved that 11, 33, 55, 88 and 555 are the only repdigits that can be expressed as difference of two Fibonacci numbers and 11, 22 and 44 are the only repdigits that can be expressed as difference of two Lucas numbers. Our work closely follows K. Bhoi and P.K. Ray \cite{br2021}.

The objective of this paper is to extend this study by exploring the repdigits  expressible as difference of two balancing or Lucas-balancing numbers. For this purpose, we have considered the following two equations
	\begin{equation}\label{1}
		B_n - B_m = d\biggr(\frac{10^k -1}{9}\biggr)
	\end{equation}
	and
	\begin{equation}\label{2}
		C_n - C_m = d\biggr(\frac{10^k -1}{9}\biggr),
	\end{equation}
	with $n > m$ and $1\leq d \leq 9$.
	Assume $k \geq 2$ to avoid trivial solutions.
	
\section{Auxiliary Results}
To solve the Diophantine equations involving repdigits and terms of binary recurrence sequences, many authors have used Baker's theory to reduce lower bounds concerning linear forms in logarithms of algebraic numbers. These lower bounds play an important role in solving such Diophantine equations. We start by recalling some basic definitions and results from algebraic number theory.

	Let $\lambda$ be an algebraic number with minimal primitive polynomial 
	
		\[f(X) = a_0(X-\lambda^{(1)})\cdot\cdot\cdot(X-\lambda^{(k)}) \in \mathbb{Z}[X],\]
	
	where $a_0 > 0$ and $\lambda^{(i)}$'s are conjugates of $\lambda$. Then
	\begin{center}
		$h(\lambda) = \dfrac{1}{k}\biggr(\log a_0 + \sum\limits_{j=1}^{k}$max$\{0,\log|\lambda^{(j)}|\}\biggr)$
	\end{center}
	is called the $logarithmic$ $height$ of $\lambda$. If $\lambda = \dfrac{a}{b}$ is a rational number with $gcd(a,b)=1$ and $b>1$, then  $h(\lambda)=\log$(max$\{|a|,b\})$. We give some properties of the logarithmic height whose proofs can be found in \cite{bms2006}.
	
	Let $\gamma$ and $\eta$ be two algebraic numbers. then 
	\begin{equation*}
		\begin{split}
			& (i)\hspace{0.2cm} h(\gamma\pm\eta) \leq h(\gamma)+h(\gamma)+\log 2,\\
			&(ii)\hspace{0.2cm} h(\gamma\eta^{\pm 1})\leq h(\gamma)+h(\eta),\\
			&(iii)\hspace{0.2cm} h(\gamma ^k)=|k|h(\gamma).
		\end{split}
	\end{equation*}
	Now we give a theorem which is  deduced from  Corollary  $2.3$  of  Matveev \cite{m2000} and provides a large upper bound for the subscript $n$ in the equations \eqref{1} and \eqref{2} (also see Theorem $9.4$ in \cite{bms2006}).
	
	\begin{theorem}\label{thm1}\cite{m2000}. Let $\mathbb{L}$ be an algebraic number field of degree $d_{\mathbb{L}}$. Let $\gamma_1,\ldots,\gamma_l \in \mathbb{L}$ be positive real numbers and $b_1, \ldots, b_l$ be nonzero integers. If $\Gamma = \prod\limits_{i=1}^{l} \gamma_{i}^{b_i} -1$ is not zero, then 
	\begin{equation*}
		\log|\Gamma| > -1.4\cdot 30^{l+3}\cdot l^4.5 \cdot d_{\mathbb{L}}^2 (1+\log d_{\mathbb{L}}^2)(1+\log(D))A_1 A_2 \cdot\cdot\cdot A_l,
	\end{equation*}
	where $D\geq$ max$\{|b_1|,\ldots,|b_l|\}$ and  $A_1, \cdots, A_l$ are positive integers such that $A_j \geq h'(\gamma_j)$ = $max\{d_{\mathbb{L}}h(\gamma_j),|\log \gamma_j|, 0.16\},$ for $j= 1,\ldots,l$.
	\end{theorem}
	
	The following result of Baker-Davenport due to Dujella and Peth\H{o} \cite{dp1998} is another tool in our proofs. It will be used to reduce the upper bounds of the variables on \eqref{1} and \eqref{2}.
	\newline
	
	\begin{lemma}\label{lem1}\cite{dp1998}. Let $M$ be a positive integer and $\frac{p}{q}$ be a convergent of the continued fraction of the irrational number $\tau$ such that $\tau > 6M$. Let $A, B, \mu$ be some real numbers with $A > 0$ and $B > 1$. Let $\epsilon := \left\lVert \mu q\right\rVert - M\left\lVert \tau q\right\rVert$, where $\left\lVert .\right\rVert$ denotes the distance from the nearest integer. If $\epsilon > 0$, then there exists no solution to the inequality
	\begin{equation}
		0 <|u\tau - v +\mu|<AB^{-w},
	\end{equation}
	in positive integers $u,v,w$ with $u\leq M$ and $w\geq \dfrac{\log(Aq/\epsilon)}{\log B}$.
	\end{lemma}

	We conclude this section by recalling the following lemma that we need in the sequel:

	\begin{lemma}\label{lem2}\cite{sl2014}. Let $r\geq 1$ and $H > 0$ be such that $H > (4r^2)^r$ and $H>L/(\log L)^r$. Then $L < 2^rH(\log H)^r$.
	\end{lemma}

\section{Repdigits as difference of two balancing numbers}
	\begin{theorem}\label{thm2} There are no non-trivial repdigits which can be expressed as the difference of two balancing numbers.\\
	\end{theorem}
	
\begin{proof} A brief computer search assures that there is no solution of \eqref{1} in the range $n\in\{1,50\}$, i.e., there is no repdigit as difference of two balancing numbers for $n\leq50$.
	
So, assume that $n>50$. Using Binet's formula for the balancing number sequence, \eqref{1} can be written as\\
	\begin{equation}\label{3}
		\frac{\alpha^{n}-\beta^{n}}{4\sqrt{2}} - \frac{\alpha^{m}-\beta^{m}}{4\sqrt{2}} = d\biggr(\frac{10^{k}-1 }{9}\biggr),
	\end{equation}
	\newline which further implies $\frac{\alpha^{n}}{4\sqrt{2}} - \frac{d10^k}{9} = \frac{\alpha^{m}}{4\sqrt{2}} + \frac{\beta^{n}}{4\sqrt{2}} - \frac{\beta^{m}}{4\sqrt{2}} - \frac{d}{9}$. Taking absolute values on both sides, we obtain\\
	\begin{equation*}
		\biggr|\frac{\alpha^{n}}{4\sqrt{2}} - \frac{d10^k}{9}\biggr| \leq \biggr|\frac{\alpha^{m}}{4\sqrt{2}}\biggr| + 3 \leq \frac{4\alpha^{m}}{4\sqrt{2}}.
	\end{equation*} 
	Dividing both sides of the above inequality by $\frac{\alpha^{n}}{4\sqrt{2}}$ yields\\
	\begin{equation}\label{4}
		{\biggr|1- \alpha^{-n}10^{k}\biggr(\frac{4d\sqrt{2}}{9}\biggr)\biggr|} <\frac{4}{\alpha^{n-m}}.
	\end{equation}
	We set $\Gamma_1=1-\alpha^{-n}\biggr(\dfrac{4d\sqrt{2}}{9}\biggr)$. We need to show $\Gamma_1 \neq 0$. On the contrary, suppose $\Gamma_1=0$, then $\alpha^{2n}= \frac{32 d^{2}10^{2k}}{81} \in \mathbb{Q}$, which is a contradiction to the fact that $\alpha^{n}$ is irrational for any $n > 0$. Therefore, $\Gamma_1 \neq 0$. \newline Take 
	\begin{center}
		
		$\lambda_1$ = $\alpha$, $\lambda_2$ = 10, $\lambda_3$ = $\dfrac{4d\sqrt{2}}{9}$, $b_1 = -n$, $b_2 = k$, $b_3 = 1, l = 3$, \end{center} where $\lambda_1$, $\lambda_2$, $\lambda_3$ $\in \mathbb{Q}$ and $b_1, b_2, b_3 \in \mathbb{Z}$. Observe that $\mathbb{Q}(\lambda_1, \lambda_2, \lambda_3) = \mathbb{Q}(\alpha)$, so $d_\mathbb{L}$ = 2.
	Since $k < n$, we take $D$ = max$\{n, k, 1\} = n$. By the properties of the absolute $logarithmic$ $height$, the logarithmic heights of $\lambda_1$, $\lambda_2$ and $\lambda_3$ are calculated as $h(\lambda_1)=\frac{\log\alpha}{2}$, $h(\lambda_2)= \log 10$ and $h(\lambda_3)\leq h(4d\sqrt{2})+h(9) < 6.2$.
	\newline 
	Thus,
	\begin{center}
		max$\{2h(\lambda_1),|\log \lambda_1|, 0.16 \} = \log \alpha = A_1$,\end{center}
	\begin{center}
		max$\{2h(\lambda_2),|\log \lambda_2|, 0.16 \} = 2\log 10 = A_2$,\end{center}
	\begin{center}
		max$\{2h(\lambda_3),|\log \lambda_3|, 0.16 \} < 12.4 = A_3$.\end{center}
	Applying Theorem \ref{thm1}, we get a lower bound for $\log|\Gamma_1|$ as \begin{equation*}
		\log|\Gamma_1| > -1.4\cdot30^6\cdot3^{4.5}\cdot2^2 (1+\log 2)(1+\log n)(\log\alpha)(2\log 10)(12.4).\end{equation*}
	Comparing the above inequality with \eqref{4} leads to 
	\begin{equation}\label{5}
		(n-m)\log(\alpha) < \log 4 + 9.8 \cdot 10^{13} (1+\log n) < 9.9 \cdot 10^{13} (1+\log n).
	\end{equation}
	After a rearrangement of  equation \eqref{3}, we get 
	\begin{center}
		$\dfrac{\alpha^{n}}{4\sqrt{2}} - \dfrac{\alpha^{m}}{4\sqrt{2}} - \dfrac{d10^k}{9}$ =  $\dfrac{\beta^{n}}{4\sqrt{2}} - \dfrac{\beta^{m}}{4\sqrt{2}} - \dfrac{d}{9}.$
		
	\end{center}
	Taking absolute values on both sides, we obtain\\
	\begin{equation*}
		\biggr|\frac{\alpha^{n}}{4\sqrt{2}} - \frac{\alpha^{m}}{4\sqrt{2}} - \frac{d10^k}{9}\biggr| \leq 3.
	\end{equation*} 
	Dividing both sides of the above inequality by $\frac{\alpha^{n}}{4\sqrt{2}}$$(1-\alpha^{m-n})$, we get
	\begin{equation}\label{6}
		{\biggr|1- \alpha^{-n}10^{k}\biggr(\frac{4d\sqrt{2}}{9(1-\alpha^{m-n})}\biggr)\biggr|} <\frac{4}{\alpha^{n}}.
	\end{equation}
	We set $\Gamma_2$ = $1- \alpha^{-n}10^{k}\biggr(\dfrac{4d\sqrt{2}}{9(1-\alpha^{m-n})}\biggr)$.
	\newline Similarly, one can check that $\Gamma_2 \neq 0$. Thus, we have h($\lambda_1$)=h($\alpha$)=$\frac{\log\alpha}{2}$ and h($\lambda_2$)=h(10)=$\log 10$. Let $\lambda_3$ = $\biggr(\dfrac{4d\sqrt{2}}{9(1-\alpha^{m-n})}\biggr)$. Then, 
	\begin{equation*}
		\begin{split}
			h(\lambda_3) & \leq h(4d\sqrt{2}) + h(9(1-\alpha^{m-n}))\\
			&  \leq 2\log 9+\log 4+\log(\sqrt{2})+\log 2+(n-m)\frac{\log\alpha}{2}\\
			&  < 6.9+ (n-m) \frac{\log(\alpha)}{2}.
		\end{split}
	\end{equation*}
	Hence, from \eqref{5} we obtain 
	\begin{center}
		$h(\lambda_3) < 5\cdot 10^{13}(1+\log n)$.
	\end{center}
	Thus, $A_3 = 10 \cdot 10^{13}(1+\log n)=10^{14}(1+\log n)$. 
	By virtue of Theorem \ref{thm1},
	 $$\log|\Gamma_2| > -1.4\cdot30^6\cdot3^{4.5}\cdot2^2 (1+\log 2)(1+\log n)(\log\alpha)(2\log 10)(10^{14}(1+\log n)).$$
	Comparing the above inequality with \eqref{6} gives 
	\begin{equation*}
		n\log\alpha < \log 4 + 7.9 \cdot 10^{26} (1+\log n)^2 < 8 \cdot 10^{26} (1+\log n)^2.
	\end{equation*}
	With the notations of Lemma \ref{lem2}, we take $r=2, L=2, H= \frac{8\cdot10^{26}}{\log \alpha}$ to get 
	\begin{center}
		$ n <$ $2^2$ $\biggr(\dfrac{8\cdot10^{26}}{\log\alpha}\biggr)$ $\biggr(\log\biggr(\dfrac{8\cdot10^{26}}{\log \alpha}\biggr)\biggr)^2$. 
	\end{center}
	Hence, a computer search with {\it Mathematica} gives  $n<6.9\cdot10^{30}$.
	 Now, let us try to reduce the upper bound on $n$ by using the Baker-Davenport reduction method due to Dujella and Peth\H{o}\cite{dp1998}. Let 
$$\Lambda_1 = -n\log\alpha + k\log 10 + \log\biggr(\dfrac{4d\sqrt{2}}{9}\biggr).$$
	The inequality \eqref{4} can be written as 
	\begin{center}
		$|e^{\Lambda_1} - 1| < \dfrac{4}{\alpha^{n-m}}$.
	\end{center}
	Observe that $\Lambda_1 \neq 0$ as $e^{\Lambda_1} - 1 = \Gamma \neq 0$.
	Assuming $n-m\geq 2$, the right hand side in the above inequality is at most $\frac{4}{(3+2\sqrt{2})^2}<\frac{1}{2}$. The inequality $|e^{z} - 1| < y$ for real values of $z$ and $y$ implies $z < 2y$. Thus, we get $|\Lambda_1| < \frac{8}{\alpha^{n-m}}$, which implies that
	\begin{center}
		
		$\biggr|-n\log\alpha + k\log 10 + \log\biggr(\dfrac{4d\sqrt{2}}{9}\biggr) \biggr| < \dfrac{8}{\alpha^{n-m}}$.
	\end{center}
	Dividing both sides of the above inequality by $\log \alpha$, we obtain
	
	\begin{equation}\label{7}
		\biggr|k\biggr(\frac{\log 10}{\log\alpha}\biggr) - 2n + \frac{\log(\frac{4d\sqrt{2}}{9})}{\log\alpha}\biggr| <\frac{5}{\alpha^{n-m}}.
	\end{equation}
	To apply Lemma \ref{lem1}, let 
	\begin{center}
		$u=k$, $\tau=\dfrac{\log 10}{\log\alpha}$, $v = n$, $\mu= \biggr(\dfrac{\log(\frac{4d\sqrt{2}}{9})}{\log\alpha}\biggr)$, $A= 5, B= \alpha, w = n-m$.
	\end{center}
	
	We can take $M=6.9\cdot10^{30}$ which is an upper bound on $u$. The denominator of $62^{th}$ convergent of $\tau$ is $q_{62}$= 82660367338512336905381670798737, which exceeds $6M$. Considering the fact that $1\leq d\leq 9$, a quick computation with {\it Mathematica} gives us the inequality $0 < \epsilon := \left\lVert \mu q_{62}\right\rVert - M\left\lVert \tau q_{62}\right\rVert = 0.243566$. Applying Lemma \ref{lem1} to the inequality \eqref{7} for $1\leq d\leq 9$, we get $n-m\leq 43$.
	
Now, for $n-m\leq 43$, put 
	\begin{center}
		$\Lambda_2 = -n\log\alpha + k\log 10 + \log\biggr(\dfrac{4d\sqrt{2}}{9(1-\alpha^{m-n})}\biggr)$. 
	\end{center}
	The inequality \eqref{6} can be written as 
	\begin{center}
		$|e^{\Lambda_2} - 1| < \dfrac{4}{\alpha^{n}}$.
	\end{center}
	Observe that $\Lambda_2 \neq 0$ as $e^{\Lambda_2} - 1 = \Gamma \neq 0$.
	Assuming $n\geq 2$, the right hand side in the above inequality is at most $\frac{1}{2}$. The inequality $|e^{z} - 1| < y$ for real values of $z$ and $y$ implies $z < 2y$. Thus, we get $|\Lambda_2| < \frac{8}{\alpha^{n}}$, which implies that
	\begin{center}
$\biggr|-n\log\alpha + k\log 10 + \log\biggr(\dfrac{4d\sqrt{2}}{9}\biggr) \biggr| < \dfrac{8}{\alpha^{n}}$.
	\end{center}
	Dividing both sides of the above inequality by $\log \alpha$ gives
	
	\begin{equation}\label{8}
		\biggr|k\biggr(\frac{\log 10}{\log\alpha}\biggr) - n + \dfrac{\log (\frac{4d\sqrt{2}}{9(1-\alpha^{m-n})})}{\log\alpha}\biggr| <\dfrac{5}{\alpha^{n}}.
	\end{equation}
	Let  \begin{center}  
		$u=k$, $\tau=\dfrac{\log 10}{\log\alpha}$, $v = n$, $\mu=\dfrac{\log (\frac{4d\sqrt{2}}{9(1-\alpha^{m-n})})}{\log\alpha} $, $A= 5, B= \alpha, w = n$.\end{center}
	Choose $M = 6.9\cdot 10^{30}$. We find $q_{64}$=193515224029707700321265026524859 , the denominator of $64^{th}$ convergent of $\tau$ exceeds $6M$ with 0 $< \epsilon := \left\lVert \mu q_{64}\right\rVert - M\left\lVert \tau q_{64}\right\rVert$ = 0.1734988. Applying Lemma \ref{lem1} to the inequality \eqref{8} for $1\leq d\leq 9$, we get $n\leq 44$. This contradicts the assumption that $n > 50$.
	\end{proof}
	
\section{Repdigits as difference of two Lucas-balancing numbers}

\begin{theorem}\label{thm3} There are no non-trivial repdigits which can be expressed as the difference of two Lucas-balancing numbers.
\end{theorem}
	
\begin{proof} The proof is similar to that of Theorem \ref{thm2}. Using $Mathematica$, it is easy to see that there is no repdigit as difference of two Lucas-balancing numbers for $n\leq50$. So, assume that $n>50$. Let
	\begin{equation}\label{9}
		\frac{\alpha^{n}-\beta^{n}}{2} - \frac{\alpha^{m}-\beta^{m}}{2} = d\biggr(\frac{10^{k}-1 }{9}\biggr),
	\end{equation}
	 which further implies $\frac{\alpha^{n}}{2} - \frac{d10^k}{9} = \frac{\alpha^{m}}{2} + \frac{\beta^{n}}{2} - \frac{\beta^{m}}{2} - \frac{d}{9}$. Taking absolute values on both sides, we obtain\\
	\begin{equation*}
		\biggr|\frac{\alpha^{n}}{2} - \frac{d10^k}{9}\biggr| \leq \biggr|\frac{\alpha^{m}}{2}\biggr| + 2 \leq \frac{3\alpha^{m}}{2}.
	\end{equation*} 
	Dividing both sides of the above inequality by $\frac{\alpha^{n}}{2}$ yields\\
	\begin{equation}\label{10}
		{\biggr|1- \alpha^{-n}10^{k}\biggr(\frac{2d}{9}\biggr)\biggr|} <\frac{3}{\alpha^{n-m}}.
	\end{equation}
	Take $\Gamma_3$ = $1-\alpha^{-n}10^{k}\bigg(\dfrac{2d}{9}\biggr)$. We need to show $\Gamma_3 \neq 0$. On the contrary, suppose $\Gamma_3=0$, then $\alpha^{2n}= \frac{4 d^{2}10^{2k}}{81} \in \mathbb{Q}$, which is a contradiction to the fact that $\alpha^{n}$ is irrational for any $n > 0$. Therefore, $\Gamma_3 \neq 0$.
	
	 Take 
	\begin{center}
		$\lambda_1$ = $\alpha$, $\lambda_2 = 10$, $\lambda_3$ = $\dfrac{2d}{9}$, $b_1 = -n$, $b_2 = k$, $b_3 = 1, l = 3$,
	\end{center}  where $\lambda_1$, $\lambda_2$, $\lambda_3$ $\in \mathbb{Q}$ and $b_1, b_2, b_3 \in \mathbb{Z}$. Observe that $\mathbb{Q}(\lambda_1, \lambda_2, \lambda_3) = \mathbb{Q}(\alpha)$, so $d_\mathbb{L} = 2$.
	Since $k < n$, we take $D=$max$\{n, k, 1\} = n$. The logarithmic heights of $\lambda_1$, $\lambda_2$ and $\lambda_3$ are calculated as $h(\lambda_1)=\frac{\log\alpha}{2}$, $h(\lambda_2)= \log 10$ and $h(\lambda_3)\leq h(4d\sqrt{2})+h(9) < 5.1$ .
	\newline Thus,
	\begin{center}
		max$\{2h(\lambda_1),|\log \lambda_1|, 0.16 \} = \log \alpha$ = $A_1$,
		max$\{2h(\lambda_2),|\log \lambda_2|, 0.16 \} = 2\log 10 = A_2$,
		max$\{2h(\lambda_3),|\log \lambda_3|, 0.16 \} < 10.2 = A_3$.\end{center}
	Now, applying Theorem \ref{thm1}, we get a lower bound for $\log|\Gamma_3|$ as \begin{equation*}
		\log|\Gamma_3| > -1.4\cdot30^6\cdot3^{4.5}\cdot2^2 (1+\log 2)(1+\log n)(\log\alpha)(2\log 10)(10.2).\end{equation*}
	Comparing the above inequality with \eqref{10}, gives 
	\begin{equation}\label{11}
		(n-m)\log(\alpha) < \log 3 + 8.1 \cdot 10^{13} (1+\log n) < 8.2 \cdot 10^{13} (1+\log n).
	\end{equation}
	Rearranging equation \eqref{9}, we get
	\begin{center}
		$\dfrac{\alpha^{n}}{2} - \dfrac{\alpha^{m}}{2} - \dfrac{d10^k}{9} =  \dfrac{\beta^{n}}{2} - \dfrac{\beta^{m}}{2} - \dfrac{d}{9}$.
		
	\end{center}
	Taking absolute values on both sides of the above inequality, we obtain\\
	\begin{equation*}
		\biggr|\frac{\alpha^{n}}{2} - \frac{\alpha^{m}}{2} - \frac{d10^k}{9}\biggr| \leq 2.
	\end{equation*} 
	Dividing both sides of the above inequality by $\frac{\alpha^{n}}{2}$$(1-\alpha^{m-n})$ implies
	\begin{equation}\label{12}
		{\biggr|1- \alpha^{-n}10^{k}\biggr(\frac{2d}{9(1-\alpha^{m-n})}\biggr)\biggr|} <\frac{3}{\alpha^{n}}.
	\end{equation}
	Take $\Gamma_4$ = $1- \alpha^{-n}10^{k}\biggr(\dfrac{2d}{9(1-\alpha^{m-n})}\biggr)$.
	\newline In a similar manner, one can check that $\Gamma_4 \neq 0$. Here we have $h(\lambda_1)=h(\alpha)=\frac{\log\alpha}{2}$ and $h(\lambda_2)=h(10)=\log 10$. Let $\lambda_3$ = $\biggr(\dfrac{2d}{9(1-\alpha^{m-n})}\biggr)$. Then, 
	\begin{equation*}
		\begin{split}
			h(\lambda_3) & \leq h(2d) + h(9(1-\alpha^{m-n}))\\
			& \leq 2\log 9+2\log 2+(n-m)\frac{\log\alpha}{2}\\
			& < 5.8+ (n-m) \frac{\log(\alpha)}{2}.
		\end{split}
	\end{equation*}
	Hence, from \eqref{11} we obtain 
	\begin{center}
		$h(\lambda_3) < 4.2\cdot 10^{13}(1+\log n)$.
	\end{center}
	Thus, we take $A_3$ = 8.4 $\cdot 10^{13}(1+\log n)$.
	\newline\newline By virtue of Theorem \ref{thm1}
	\begin{equation*}
		\begin{split}
			\log|\Gamma_4| >  -& 1.4\cdot30^6\cdot3^{4.5}\cdot2^2 (1+\log 2)(1+\log n)(\log\alpha)(2\log 10)\\
			& \times(8.4\cdot 10^{13})(1+\log n)
		\end{split}
	\end{equation*}
	\newline Comparing the above inequality with \eqref{12} gives 
	\begin{equation*}
		n\log\alpha < \log 3 + 6.7 \cdot 10^{26} (1+\log n)^2 < 6.8 \cdot 10^{26} (1+\log n)^2.
	\end{equation*}
	With the notations of Lemma \ref{lem2}, we take $r=2, L=2, H= \frac{6.8\cdot10^{26}}{\log \alpha}$ to get 
	\begin{center}
		n $<$ $2^2$ $\biggr(\dfrac{6.8\cdot10^{26}}{\log\alpha}\biggr)$ $\biggr(\log\biggr(\dfrac{6.8\cdot10^{26}}{\log \alpha}\biggr)\biggr)^2$ $<$ $5.8\cdot10^{30}$.
	\end{center}
	
	Now, we reduce the bound by using the Baker-Davenport reduction method due to Dujella and Peth\H{o} \cite{dp1998}. Let 
	\begin{center}
		$\Lambda_3 = -n\log\alpha + k\log 10 + \log\biggr(\dfrac{2d}{9}\biggr)$ .
	\end{center}
	The inequality \eqref{10} can be written as 
	\begin{center}
		$|e^{\Lambda_3} - 1| < \dfrac{3}{\alpha^{n-m}}$.
	\end{center}
	Observe that $\Lambda_3 \neq 0$ as $e^{\Lambda_3} - 1 = \Gamma \neq 0$.
	Assuming $n-m\geq 2$, the right hand side in the above inequality is at most $\frac{3}{(3+2\sqrt{2})^2}<\frac{1}{2}$. The inequality $|e^{z} - 1| < y$ for real values of $z$ and $y$ implies $z< 2y$. Thus, we get $|\Lambda_3| < \frac{6}{\alpha^{n-m}}$, which implies that
	\begin{center}
		
		$\biggr|-n\log\alpha + k\log 10 + \log\biggr(\dfrac{2d}{9}\biggr) \biggr| < \dfrac{6}{\alpha^{n-m}}$.
	\end{center}
	Dividing both sides of the above inequality by $\log \alpha$, we obtain
	
	\begin{equation}\label{13}
		\biggr|k\biggr(\frac{\log 10}{\log\alpha}\biggr) - 2n + \frac{\log(\frac{2d}{9})}{\log\alpha}\biggr| <\frac{4}{\alpha^{n-m}}.
	\end{equation}
	To apply Lemma \ref{lem1}, let 
	\begin{center}
		$u=k, \tau=\dfrac{\log 10}{\log\alpha}, v = n, \mu= \biggr(\dfrac{\log(\frac{2d}{9})}{\log\alpha}\biggr), A= 4, B= \alpha, w = n-m$.
	\end{center}
	
	We can take $M= 5.8\cdot10^{30}$ which  is  an  upper  bound  on $u$. The denominator of $62^{th}$ convergent of $\tau$ is $q_{62}$= 82660367338512336905381670798737, which exceeds $6M$. Considering the fact that $1\leq d\leq 9$, a quick computation with $Mathematica$ gives us the inequality $0 < \epsilon := \left\lVert \mu q_{60}\right\rVert - M\left\lVert \tau q_{62}\right\rVert = 0.0781826$. Applying Lemma \ref{lem1} to the inequality \eqref{13} for $1\leq d\leq 9$, we get $n-m\leq 43$. Now, for $n-m\leq 43$, put 
	\begin{center}
		$\Lambda_4 = -n\log\alpha + k\log 10 + \log\biggr(\dfrac{2d}{9(1-\alpha^{m-n})}\biggr)$. 
	\end{center}
	The inequality \eqref{12} can be written as 
	\begin{center}
		$|e^{\Lambda_4} - 1| < \dfrac{3}{\alpha^{n}}$.
	\end{center}
	Observe that $\Lambda_4 \neq 0$ as $e^{\Lambda_4} - 1 = \Gamma \neq 0$.
	Assuming $n\geq 2$, the right hand side in the above inequality is at most $\frac{1}{2}$. The inequality $|e^{z} - 1| < y$ for real values of $z$ and $y$ implies $z < 2y$. Thus, we get $|\Lambda_4| < \frac{6}{\alpha^{n}}$, which implies that
	\begin{center}
		
		$\biggr|-n\log\alpha + k\log 10 + \log\biggr(\dfrac{2d}{9}\biggr) \biggr| < \dfrac{6}{\alpha^{n}}$.
	\end{center}
	Dividing both sides of the above inequality by $\log \alpha$ gives
	
	\begin{equation}\label{14}
		\biggr|k\biggr(\frac{\log 10}{\log\alpha}\biggr) - n + \dfrac{\log (\frac{2d}{9(1-\alpha^{m-n})})}{\log\alpha}\biggr| <\frac{4}{\alpha^{n}}.
	\end{equation}
	Let \begin{center}
		$u=k, \tau=\dfrac{\log 10}{\log\alpha}, v = n, \mu= \biggr(\dfrac{\log(\frac{2d}{9(1-\alpha^{m-n})})}{\log\alpha}\biggr), A= 4, B= \alpha, w = n$.
	\end{center}
	Choose $M = 5.8\cdot 10^{30}$. We find $q_{65}= 497885304750610764058413408775840 $, the denominator of $65^{th}$ convergent of $\tau$ exceeds $6M$ with 0 $< \epsilon := \left\lVert \mu q_{65}\right\rVert - M\left\lVert \tau q_{65}\right\rVert = 0.0041201 $. Applying Lemma \ref{lem1} to the inequality \eqref{14} for $1\leq d\leq 9$, we get $n\leq 46$. This contradicts our assumption that $n > 50$.
	\end{proof}	
	
{\bf Data Availability Statements:} Data sharing not applicable to this article as no datasets were generated or analyzed during the current study.

{\bf Funding:} The authors declare that no funds or grants were received during the preparation of this manuscript.

{\bf Declarations:}

{\bf Conflict of interest:} On behalf of all authors, the corresponding author states that there is no Conflict of interest.

\end{document}